\documentclass[12pt]{amsart}
\usepackage{amssymb,amsmath}
\usepackage{enumerate}
\usepackage{mathrsfs}

\newtheorem{theorem}{Theorem}[section]

\newtheorem{corollary}[theorem]{Corollary}
\theoremstyle{definition}
\newtheorem{definition}[theorem]{Definition}
\newtheorem{example}[theorem]{Example}

\newcommand{\clb}{\mathscr{B}}

\newcommand{\raro}{\rightarrow}

\newcommand{\vp}{\varphi}

\newcommand{\bd}{\mathbb{D}}

\newcommand{\bc}{\mathbb{C}}

\title[A Bishop-Phelps-Bollob\'{a}s theorem for $A(\mathbb{D})$]{A Bishop-Phelps-Bollob\'{a}s theorem for disc algebra}

\author[Bala]{Neeru Bala}
\address{Indian Statistical Institute, Statistics and Mathematics Unit, 8th Mile, Mysore Road, Bangalore, 560 059, India}
\email{neerusingh41@gmail.com }

\author[Sarkar]{Jaydeb Sarkar}
\address{Indian Statistical Institute, Statistics and Mathematics Unit, 8th Mile, Mysore Road, Bangalore, 560 059, India}
\email{jaydeb@gmail.com, jay@isibang.ac.in}

\author[Sensarma]{Aryaman Sensarma}
\address{Indian Statistical Institute, Statistics and Mathematics Unit, 8th Mile, Mysore Road, Bangalore, 560 059, India}
\email{aryamansensarma@gmail.com}


\numberwithin{equation}{section}

\begin{document}
	
\begin{abstract}
Let $\mathbb{D}$ represent the open unit disc in $\mathbb{C}$. Denote by $A(\mathbb{D})$ the disc algebra, and $\mathscr{B}(X, A(\mathbb{\mathbb{D}}))$ the Banach space of all bounded linear operators from a Banach space $X$ into $A(\mathbb{D})$. We prove that, under the assumption of equicontinuity at a point in $\partial \mathbb{D}$, the Bishop-Phelps-Bollob\'{a}s property holds for $\mathscr{B}(X, A(\mathbb{\mathbb{D}}))$.
\end{abstract}
	

\maketitle

\section{Introduction}\label{sec: intro}

Norm attainment is one of the most natural properties that bounded linear operators or functionals acting on Banach spaces can have. Given that this property is automatic for functionals or operators acting on finite-dimensional Banach spaces but not for those acting on infinite-dimensional Banach spaces, one is tempted to wonder about the behavior of norm-attaining functionals or operators acting on infinite-dimensional Banach spaces (see the survey \cite{Aron 3}). Recall that a bounded linear operator $T : X \raro Y$ between Banach spaces $X$ and $Y$ ($T \in \clb(X, Y)$ in short, and $T \in \clb(X)$ if $Y = X$) is \textit{norm attaining} if there exists a unit vector $x_0 \in S_X$ such that
\[
\|T\|_{\clb(X, Y)} =	\|T x_0\|_Y,
\]
where $S_X = \{x \in X: \|x\|=1\}$, the unit sphere of $X$. All the Banach spaces considered in this paper are over the field $\mathbb{C}$.
	
In reference to the nature of norm-attaining functionals on Banach spaces, the classical Bishop-Phelps theorem states \cite{BP}:

\begin{theorem}[Bishop-Phelps]
The set of norm-attaining functionals on a Banach space is norm dense in the dual space.
\end{theorem}
	
Bollob\'{a}s sharpened this amazing result, which gives a simultaneous approximation of functionals close to norm attainment at a point by norm-attaining functionals while keeping the point of norm attainment close to the given point. More specifically \cite[Corollary 3.3]{CASCALES}:

\begin{theorem}[Bishop-Phelps-Bollob\'{a}s]\label{BPB}
Let $X$ be a Banach space, $x\in S_X$, $f\in S_{X^*}$, and let $\epsilon \in (0, 1)$. Suppose
\[
|f(x)| > 1-\epsilon.
\]
Then there exist $y \in S_X$ and $g \in S_{X^*}$ such that
\begin{enumerate}
\item $|g(y)|=1$,
\item $\|x-y\|<\sqrt{2\epsilon}$, and
\item $\|f-g\|<\sqrt{2\epsilon}$.
\end{enumerate}
\end{theorem}

Now we will talk about the follow-up question that comes up naturally: norm-attainment for operators on Banach spaces. Lindenstrauss \cite{LINDENSTRAUSS} pioneered the study of norm-attainment or Bishop-Phelps type property for operators between Banach spaces. And finally, Acosta, Aron, Garc\'{i}a, and Maestre \cite[Definition 1.1]{ACOSTA} introduced the notion of Bishop-Phelps-Bollob\'{a}s property for operators on Banach spaces.
	
\begin{definition}\label{def: BPB operators}
A pair of Banach spaces $(X,Y)$ satisfy the Bishop-Phelps-Bollob\'{a}s property, if for every $\epsilon>0$, there exist $\beta(\epsilon) > 0$
and $\gamma(\epsilon)>0$ with  $\lim_{t \raro 0} \beta(t)= 0$ such that for all $T\in S_{\clb(X,Y)}$ and $x \in S_X$ with
\[
\|Tx\|> 1 - \gamma,
\]
then there exist $y \in S_X$ and $N \in S_{\mathscr{B}(X,Y)}$ such that
\begin{enumerate}
\item $\|N y\|=1$,
\item $\|x - y\|<\beta$, and
\item $\|T - N\|<\epsilon$.
\end{enumerate}
If $X = Y$, then we simply say that $X$ satisfies the Bishop-Phelps-Bollob\'{a}s property.
\end{definition}

Following the case of functionals on Banach spaces, one might wonder if the Bishop-Phelps-Bollob\'{a}s theorem also applies to $\clb(X,Y)$ for Banach spaces $X$ and $Y$. However, it follows from the same paper of Lindenstrass \cite{LINDENSTRAUSS} that an arbitrary pair of Banach spaces does not satisfy the Bishop-Phelps-Bollob\'{a}s property. Since then, identifying Banach spaces that satisfy the Bishop-Phelps-Bollob\'{a}s property has proven to be a challenging as well as important problem.

Let us highlight some of the noteworthy results in this direction. In \cite{ACOSTA}, Acosta, Aron, García, and Maestre presented a condition on $Y$ for which $(X, Y)$ satisfy the Bishop-Phelps-Bollob\'{a}s property for arbitrary Banach space $X$. In particular, if $Y$ is $c_0$ or $\ell^{\infty}$, then for arbitrary Banach space $X$, the pair $(X,Y)$ satisfy the Bishop-Phelps-Bollob\'{a}s property. If $X=C(K)$ and $Y=C(M)$ for some compact Hausdorff spaces $K$ and $M$, then $(X, Y)$ satisfy the Bishop-Phelps-Bollob\'{a}s property in the real case \cite{ACOSTA 2014} (the complex case, however, is still an open question). In this regard, also see Johnson and Wolfe \cite{JOHNSON}. Furthermore, in \cite{CASCALES,Kim}, it is proved that, if $Y$ is a subalgebra of $C(K)$ or $C_b(K)$, then for arbitrary Banach space $X$, the pair $(X,Y)$ satisfy the Bishop-Phelps-Bollob\'{a}s property for Asplund operators. Another example is the space $H^\infty(\bd)$ \cite{Bala}, which will be discussed in the following paragraph. We refer the reader to \cite{Aron 1, Carando} for more on Bishop-Phelps-Bollob\'{a}s in different contexts.

Now we turn to disc algebra, the central object of this paper. The disc algebra $A(\mathbb{D})$ is the commutative Banach algebra of all bounded analytic functions on the open unit disc $\bd$ with a continuous extension on $\overline{\mathbb{D}}$. Therefore
\[
A(\bd) = \text{Hol}(\mathbb{D}) \cap C(\overline{\bd}).
\]
In this context, we recall that $A(\bd)$ is a closed subalgebra of $H^\infty(\bd)$, the commutative Banach algebra of bounded analytic functions on $\bd$. The space $H^\infty(\bd)$ (as well as its closed subalgebra $A(\bd)$) is equipped with the uniform norm
\[
\|\vp\|_\infty = \sup_{z \in \bd} |\vp(z)| \qquad (\vp \in H^\infty(\bd)).
\]
There are many reasons to study disc algebra. In fact, disc algebra is one of the most important and concrete non-selfadjoint commutative Banach algebras \cite{B1,B2,FT}. Since this space is more ``tractable'' (than, say, $H^\infty(\bd)$), it can be used to test a theory of interest. Another factor that intrigued our interest in disc algebra had been a recent paper \cite{Bala} that revealed that the Bishop-Phelps-Bollob\'{a}s property holds true for $H^\infty(\bd)$. Clearly, $H^\infty(\bd)$ is another concrete example of non-selfadjoint operator algebra. Finding a follow-up solution for disc algebra appears to be a pressing issue, which is exactly what we do in this paper. However, we have not yet been able to fully settle the question for $A(\bd)$ in this circumstance, despite the fact that a priori we ``know'' that $H^\infty(\bd)$ is a more complex object than $A(\bd)$. We assume in addition that the image of the unit sphere of the given operator satisfies the equicontinuity property at a point in $\mathbb{T}$.

Therefore, the main contribution of this paper is the fact that under the assumption of equicontinuity (see Definition \ref{def: equicont} below), the Bishop-Phelps-Bollob\'{a}s property of simultaneous approximations (as defined in Definition \ref{def: BPB operators}) holds for $A(\bd)$. In fact, we prove much more: under the equicontinuity assumption, $(X, A(\mathbb{D}))$ satisfies the Bishop-Phelps-Bollob\'{a}s property for arbitrary Banach space $X$.

A similar conclusion holds for the Asplund operators on uniform algebras \cite{Aron 2, CASCALES}. However, $A(\bd)$ is not Asplund, and on the other hand, the present equicontinuity condition appears to be more practical in terms of checking the relevant conditions (cf. Section \ref{sec:final}).

The paper is structured as follows. Besides this section, there are two more sections in this paper. The first one addresses the main result, while the final section provides examples of operators that do (as well as do not) admit norms, and operators on $A(\bd)$ that satisfy the main theorem's hypothesis.

\section{The main result}

In this section, we prove the Bishop-Phelps-Bollob\'{a}s property for $A(\bd)$. We begin by recalling the definition of equicontinuity of a family of functions between Banach spaces.
	
\begin{definition}\label{def: equicont}
Let $X$ and $Y$ be two Banach spaces, and let $x_0 \in X$. A family of functions $\mathcal{F}$ from $X$ to $Y$ is said to be equicontinuous at $x_0$ if for every $\epsilon>0$, there exists $\delta>0$ such that
\[
\|f(x)-f(x_0)\|<\epsilon,
\]
for all $f \in \mathcal{F}$ and $x \in X$ such that $\|x-x_0\|<\delta$.
\end{definition}
	
Now we proceed to the main result of this paper. The starting point of the proof of the result is similar to that of the space $H^\infty(\bd)$ \cite[Lemma 2.1]{Bala} (as well as some other spaces, cf. \cite[Lemma 2.5]{CASCALES} and \cite[Lemma 3]{Kim}). However, as one progresses, the concept of the proof diverges significantly from that of previous spaces.
	
In what follows, we will adopt a specific notational convention. For each $f\in A(\bd)$, we will denote by $\tilde{f}$ the continuous extension of $f$ to all of $\bar{\mathbb{D}}$.

\begin{theorem}\label{thm}
Let $X$ be a Banach space, $x_0\in S_{X}$, $T \in S_{\mathcal{B}(X,A(\mathbb{D}))}$, $\epsilon \in (0,1)$, and let $\theta_0$ be a real number. Suppose
\[
\mathcal{F} = \{\widetilde{Tx}~~|~~ x \in S_{X}\},
\]
is equicontinuous at $e^{i\theta_0}$, and
\[
\left|\widetilde{Tx_0}(e^{i\theta_0})\right|>1-\dfrac{\epsilon}{3}.
\]
Then there exist $y_0\in S_{X}$ and $N \in S_{\mathcal{B}(X,A(\mathbb{D}))}$ such that
\begin{enumerate}
\item $\|N y_0\| =1$,
\item $\|x_0 - y_0\|< \sqrt{2\epsilon}$, and
\item $\|T - N\|<8\epsilon$.
\end{enumerate}
\end{theorem}
\begin{proof}
Since $0 < \epsilon < 1$, the Stolz region $\Omega_{\epsilon}$ is well defined, where
\[
\Omega_{\epsilon}: = \{z\in\bc:|z|+(1-\epsilon)|1-z|\leq 1\}.
\]
By the above definition, it is evident that $0, 1 \in \Omega_{\epsilon}$. Moreover, there exists a homeomorphism $\psi_{\epsilon}: \overline{\bd}\rightarrow \Omega_{\epsilon}$ (a refinement of the Riemann mapping theorem, see \cite[Theorem 14.8, 14.19]{Rudin}) such that
\begin{enumerate}
\item $\psi_{\epsilon}|_{\bd}$ is a conformal mapping onto the interior of $\Omega_{\epsilon}$,
\item $\psi_{\epsilon}(1)=1$, and
\item $\psi_{\epsilon}(0)=0$.
\end{enumerate}
By construction of $\Omega_{\epsilon}$, it follows that $\epsilon^2 \overline{\bd} \subseteq \Omega_\epsilon$. Moreover, since $0$ is in the open set ${\psi_\epsilon}^{-1}(\epsilon^2 \bd)$, there exists $\delta_1>0$ such that $\delta_1<\epsilon$ and
\begin{align}\label{Eq:1}
\delta_1 \bd \subseteq {\psi_\epsilon}^{-1}(\epsilon^2 \bd).
\end{align}
Now we deviate from the proof of \cite[Lemma 2.1]{Bala}. Since $\mathcal{F}=\{\widetilde{Tx}~~|~~ x \in S_{X}\}$ is equicontinuous at $e^{i\theta_0}$, there exists $\delta_2>0$ such that
\begin{align}\label{Eq:2}
|\widetilde{Tx}(z)-\widetilde{Tx}(e^{i\theta_0})|<\epsilon,
\end{align}
for all $z \in \overline{\mathbb{D}}$ such that $|z-e^{i\theta_0}|<\delta_2$. Define an open set $U$ in $\overline{\mathbb{D}}$ by
\[
U=\{z \in \overline{\mathbb{D}}: |z-e^{i\theta_0}|<\delta_2\}.
\]
Since $e^{i\theta_0} \in \mathbb{T}$, there exist $\tilde{g_1} \in A(\mathbb{D})$ such that (see \cite[page 80-81]{HOFFMAN})
\[
\tilde{g_1}(e^{i\theta_0})=0,
\]
and
\[
\mbox{Re}\tilde{g_1}(z)<0 \qquad (z \in \overline{\mathbb{D}}\setminus K),
\]
where
\[
K=\{e^{i\theta_0}\}.
\]
By the continuity of $\tilde{g_1}$ on $\overline{\mathbb{D}}$, we find a number $\gamma>0$ such that
\[
\mbox{Re}\tilde{g_1}(z) \leq -\gamma \qquad (z \in \overline{\mathbb{D}}\setminus U).
\]
Choose $0<\epsilon_1 < 1$ such that
\[
\epsilon_1^\gamma< \delta_1.
\]
Since $e^{-n} \raro 0$ as $n \raro \infty$, there exists $n_0\in \mathbb{N}$ such that
\[
e^{-n_0}<\epsilon_1.
\]
Define
\[
\tilde{h}(z) :=e^{n_0\tilde{g_1}(z)} \qquad (z \in \overline{\mathbb{D}}).
\]
In view of $\mbox{Re}\tilde{g_1}(z) \leq 0$, we find
\[
\begin{split}
|\tilde{h}(z)| & = e^{n_0\mbox{Re}\tilde{g_1}(z)}
\\
& \leq e^0
\\
& =1,
\end{split}
\]
whereas $\tilde{g_1}(e^{i \theta_0})=0$ yields
\[
\tilde{h}(e^{i \theta_0})=e^{n_0\tilde{g_1}(e^{i \theta_0})}=1.
\]
For each $z \in \overline{\mathbb{D}}\setminus U$, we have
\[
\begin{split}
|\tilde{h}(z)| & = e^{n_0\mbox{Re}\tilde{g_1}(z)}
\\
& =e^{{-n_0}^{-\mbox{Re}\tilde{g_1}(z)}}
\\
& <\epsilon_1^{-\mbox{Re}\tilde{g_1}(z)}.
\end{split}
\]
From the fact that $-\mbox{Re}\tilde{g_1}(z)\geq \gamma$ for all $z \in \overline{\mathbb{D}}\setminus U$, and $0<\epsilon_1<1$, we conclude that
\[
|\tilde{h}(z)|<\epsilon_1^\gamma.
\]
By the inequality $\epsilon_1^\gamma<\delta_1$, we must have
\[
|\tilde{h}(z)| < \delta_1 \qquad (z \in \overline{\mathbb{D}}\setminus U).
\]
Define $\eta$ on $\overline{\mathbb{D}}$ by
\[
\eta:=\Psi_{\epsilon} \circ \tilde{h}.
\]
Then
\[
\begin{split}
\eta(e^{i \theta_0}) & = (\Psi_{\epsilon}\circ \tilde{h})(e^{i \theta_0})
\\
& =\Psi_{\epsilon}(1)
\\
& = 1,
\end{split}
\]
and
\[
\begin{split}
\eta(\overline{\mathbb{D}}\setminus U)& = \Psi_{\epsilon}(\tilde{h}(\overline{\mathbb{D}}\setminus U))
\\
& \subseteq \Psi_{\epsilon}(\delta_1\mathbb{D})
\\
& \subseteq \epsilon^2\mathbb{D}.
\end{split}
\]
Moreover,
\[
\begin{split}
|\eta(z)|& = |\Psi_{\epsilon}(\tilde{h}(z))|
\\
& \leq |\tilde{h}(z)|
\\
& \leq 1,
\end{split}
\]
for all $z \in \overline{\mathbb{D}}$ yields $\|\eta\|\leq 1$. Also, for $z \in \overline{\mathbb{D}}$, we have
\[
\begin{split}
|\eta(z)|+(1-\epsilon)|1-\eta(z)| & = |\Psi_{\epsilon}(\tilde{h}(z))|+(1-\epsilon)|1-\Psi_{\epsilon}(\tilde{h}(z))|
\\
& \leq 1.
\end{split}
\]
Define a functional $\Psi: X \longrightarrow \mathbb{C}$ by
\[
\Psi x=\widetilde{Tx}(e^{i \theta_0}) \qquad (x \in X).
\]
It is easy to see that $\|\Psi\| \leq 1$. Also
\[
\begin{split}
|\Psi x_0|& = |\widetilde{Tx_0}(e^{i \theta_0})|
\\
& >1-\dfrac{\epsilon}{3}.
\end{split}
\]
Finally, define $\Psi_1=\dfrac{\Psi}{\|\Psi\|}: X \longrightarrow \mathbb{C}$. Then $\|\Psi_1\|= 1$, and
\[
\begin{split}
|\Psi_1 x_0| & =\left|\dfrac{\Psi}{\|\Psi\|}x_0\right|
\\
& \geq |\Psi x_0|
\\
& > 1-\dfrac{\epsilon}{3}.
\end{split}
\]
Now we can apply the Bishop-Phelps-Bollob\'{a}s theorem, Theorem \ref{BPB}, to $\Psi_1$. Consequently, there exist a vector $y_0\in S_{X}$ and a linear functional $\Psi_2: X \raro \mathbb{C}$ such that $\|\Psi_2\|=1$ and
\begin{enumerate}[(i)]
\item $|\Psi_2(y_0)|=1$,
\item $\|y_0-x_0\|\leq \sqrt{2\epsilon} $, and
\item $\left \|\Psi_1 -\Psi_2\right\|< \sqrt{2\epsilon}$.
\end{enumerate}
Define a linear operator $N:X \longrightarrow A(\mathbb{D})$ by
\begin{equation}\label{eqn: N}
(Nx)(z):=\eta(z)\Psi_2(x)+(1-\epsilon)(1-\eta(z))Tx(z),
\end{equation}
for all $x \in X$ and $z \in \bd$. Since
\begin{align*}
\left|Nx (z)\right| & \leq |\eta(z)| |\Psi_2(x)| +(1-\epsilon) |1- \eta(z)| |Tx(z)|
\\
& \leq |\eta(z)|+(1-\epsilon)|1-\eta(z)|
\\
& \leq 1,
\end{align*}
for all $z \in \bd$, it follows that $\|N\|\leq 1$. Also, note that
\begin{align*}
\left|\widetilde{Ny_0} (e^{i\theta_0})\right| & = \Big|\eta(e^{i\theta_0}) \Psi_2(y_0) +(1-\epsilon) (1- \eta(e^{i\theta_0})) \widetilde{Ty_0}(e^{i\theta_0})\Big|
\\
& = |\Psi_2(y_0)|
\\
& = 1.
\end{align*}
Fix $x \in S_{X}$. Then
\begin{align*}
\|(N-T)x\|
&= \|\eta \Psi_2(x)+(1-\epsilon)(1-\eta)Tx-Tx\|
\\
&=\|\eta \Psi_2(x)-\epsilon(1-\eta)Tx+Tx-\eta Tx-Tx\|
\\
&=\|\eta(\Psi_2(x)-Tx)-\epsilon(1-\eta)Tx\|
\\
& \leq \|\eta(\Psi_2(x)-Tx)\|+\epsilon \|1-\eta\| \|Tx\|
\\
& \leq \|\eta(\Psi_2(x)-Tx)\|+2\epsilon.
\end{align*}
Next, we estimate
\begin{align*}
\|\eta(\Psi_2 x-Tx)\|
& = \|\eta (\Psi_2x-\Psi_1x+\Psi_1x-\Psi x+\Psi x-Tx)\|\\
& \leq \|\eta\|\|\Psi_2 x-\Psi_1 x\|+\|\eta\|\|\Psi_1 x-\Psi x\|+\|\eta(\Psi x-Tx)\|\\
& \leq \sqrt{2\epsilon}+\left\|\dfrac{\Psi}{\|\Psi\|}x-\Psi x\right \|+\|\eta(\widetilde{Tx}(e^{i \theta_0})-Tx)\|\\
&\leq \sqrt{2\epsilon}+|1-\|\Psi\||+\|\eta(\widetilde{Tx}(e^{i\theta_0})-Tx)\|\\
&\leq \sqrt{2\epsilon}+\dfrac{\epsilon}{3}+\|\eta(\widetilde{Tx}(e^{i\theta_0})-Tx)\|,
\end{align*}
and, finally
\begin{align*}
\|\eta(\widetilde{Tx}(e^{i\theta_0})-\widetilde{Tx})\|
&=\sup_{z \in \overline{\mathbb{D}}}|\eta(z)||\widetilde{Tx}(e^{i \theta_0})-\widetilde{Tx}(z)|\\
&\leq \sup_{z\in \overline{\mathbb{D}}\setminus U}|\eta(z)||\widetilde{Tx}(e^{i \theta_0})-\widetilde{Tx}(z)|+\sup_{z \in U}|\eta(z)||\widetilde{Tx}(e^{i \theta_0})-\widetilde{Tx}(z)|\\
&\leq 2 \epsilon^2 +\sup_{z \in U}|\eta(z)||\widetilde{Tx}(e^{i \theta_0})-\widetilde{Tx}(z)|\\
&<2\epsilon+\epsilon.
\end{align*}
We conclude finally that $\|S-T\|<8 \epsilon$. This completes the proof.
\end{proof}

In particular, we have the following Bishop-Phelps-Bollob\'{a}s property for the disc algebra:

\begin{corollary} \label{cor}
Let $T \in \mathcal{B}(A(\mathbb{D}))$, $f_0\in S_{A(\mathbb{D})}$, $\|T\|=1$, $\epsilon \in (0,1)$, and let $\theta_0$ be a real number. Suppose
\[
\mathcal{F} = \{\widetilde{Tg}~~|~~ g \in S_{A(\mathbb{D})}\},
\]
is equicontinuous at $e^{i\theta_0}$, and
\[
\left|\widetilde{Tf_0}(e^{i\theta_0})\right|>1-\dfrac{\epsilon}{3}.
\]
Then there exist $N \in S_{\mathcal{B}(A(\mathbb{D}))}$ and $g_0\in S_{A(\mathbb{D})}$ such that
\begin{enumerate}
\item $\|N g_0\| =1$,
\item $\|f_0 - g_0\|< \sqrt{2\epsilon}$, and
\item $\|T - N\|<8\epsilon$.
\end{enumerate}

\end{corollary}	
	
We do not know how to remove the equicontinuity assumption from the preceding theorem. Furthermore, it is improbable that the equicontinuity assumption is necessary for norm attainment, but we have no evidence to support this claim either. In the following section, we will demonstrate the existence of operators that satisfy the equicontinuity condition.

\section{Examples}\label{sec:final}

Nontrivial examples of operators attaining or failing to attain norms have been of general interest. The goal of this final section is to provide some examples of operators in that range. We also draw a general observation concerning operator ideals of $\clb(A(\mathbb{D}))$.

We begin with two simple (yet classical) classes of operators that attain their norms. Let $\vp\in A(\mathbb{D})$. Define the multiplication operator $M_{\phi}$ and the composition operator $C_{\vp}$ on $A(\mathbb{D})$ by
\[
M_{\vp}f= \vp f \qquad (f \in A(\bd)),
\]
and
\[
C_{\vp}f=f\circ \vp  \qquad (f \in A(\bd)),
\]
respectively. For the case of the composition operator, we must assume in addition that $\vp$ is a self-map of $\bd$. In either case, we observe that
\[
\|M_{\vp}\|=\|\vp\|_{\infty}=\|M_{\vp}1\|,
\]
and
\[
\|C_{\vp}\|=1=\|C_{\vp}1\|.
\]
Consequently, $M_{\vp}$ and $C_{\vp}$ are norm attaining operators.

Nontrivial examples of not norm-attaining operators are based on another important space, namely, the Hardy space $H^2(\mathbb{D})$, and the disc algebra $A(\mathbb{D})$. Recall that (see \cite{Garnett, HOFFMAN})
\[
H^2(\bd) = \Big\{f \in \text{Hol}(\bd): \|f\|_2:= \Big(\sup_{0<r<1} \dfrac{1}{2\pi} \int_{\mathbb{T}} |f(rz)|^2 d\mu(z)\Big)^{\frac{1}{2}} < \infty\Big\},
\]
where $d\mu$ is the normalized Lebesgue measure on $\mathbb{T}$ ($= \partial \bd$). It is known that a function
\[
f(z) =  \sum\limits_{n=0}^{\infty} a_nz^n \in \text{Hol}(\bd),
\]
is in $H^2(\bd)$ if and only if
\[
\|f\|_2 = \Big(\sum\limits_{n=0}^{\infty} |a_n|^2 \Big)^{\frac{1}{2}}< \infty.
\]
Now, we are ready to provide a bounded linear operator $T: A(\mathbb{D}) \raro H^2(\mathbb{D})$ which is not norm-attaining.

\begin{example}
For each $f(z)=\sum_{n=0}^{\infty}a_nz^n \in H^2(\mathbb{D})$, define a linear operator $T: A(\mathbb{D}) \longrightarrow H^2(\mathbb{D})$ by
\begin{align*}
(Tf)(z)=\sum_{n=0}^\infty \left(1-\dfrac{1}{n+1}\right)a_nz^n.
\end{align*}
We compute
\begin{align*}
\|Tf\|_2 & = \left(\sum_{n=0}^{\infty}\left(1-\dfrac{1}{n+1}\right)^2|a_n|^2\right)^{\dfrac{1}{2}}
\\
& < \Big(\sum_{n=0}^{\infty}|a_n|^2\Big)^{\frac{1}{2}}
\\
& =\|f\|_2
\\
& \leq \|f\|_{\infty}.
\end{align*}
This implies that $T$ is bounded as well as not norm-attaining.
\end{example}

It is crucial to build examples that satisfy the hypothesis of this paper's main result. In the following example, we do exactly that. Note that by the James theorem \cite{J1, J2}, there exist non-norm attaining functionals in $A(\bd)^*$.

\begin{example}\label{example of disc}
Fix a functional $x^* \in A(\mathbb{D})^*$, which does not attain the norm. Also fix a nonzero $h \in A(\mathbb{D})$. Define $T \in \mathscr{B}(A(\mathbb{D}))$ by
\begin{align*}
Tf=x^*(f)h,
\end{align*}
for all $f \in A(\mathbb{D})$. A routine computation reveals that $\|T\|=\|x^*\|\|h\|$, and that $T$ is not norm-attaining. Next, we consider the collection
\[
\mathcal{F}=\{\widetilde{Tf}: f \in S_{A(\mathbb{D})}\}.
\]
By the definition of $T$, it follows that
\[
\mathcal{F} =\{x^*(f)\tilde{h}: f \in S_{A(\mathbb{D})}\}.
\]
Fix $e^{i\theta_0}\in \mathbb{T}$ and $\epsilon>0$. We claim that $\mathcal{F}$ is equicontinuous at $e^{i\theta_0}$. By continuity of $\tilde{h}$, there exists $\delta>0$ such that
\[
\|\tilde{h}(z)-\tilde{h}(e^{i\theta_0})\|<\dfrac{\epsilon}{\|x^*\|},
\]
whenever $|z-e^{i\theta_0}|<\delta$. For all $z$ such that $|z-e^{i\theta_0}|<\delta$, we have
\begin{align*}
\|\widetilde{Tf}(z)-\widetilde{Tf}(e^{i\theta_0})\|
& =\|x^*(f)\tilde{h}(z)-x^*(f)\tilde{h}(e^{i\theta_0})\|
\\
& \leq \|x^*(f)\|\|\tilde{h}(z)-\tilde{h}(e^{i\theta_0})\|
\\
& \leq \|x^*\|\|\tilde{h}(z)-\tilde{h}(e^{i\theta_0})\|
\\
&< \|x^*\|\dfrac{\epsilon}{\|x^*\|}
\\
& =\epsilon,
\end{align*}
which proves that $\mathcal{F}$ is equicontinuous at $e^{i\theta_0}\in \mathbb{T}$. Therefore, $T$ satisfies the hypotheses of Theorem \ref{thm}.
\end{example}

In particular, Theorem \ref{thm} leads us to conclude that $T$, as constructed in Example \ref{example of disc} above, can be approximated by norm-attaining operators. Additional examples along these lines include:

\begin{example}
Suppose $X$ is a non-reflexive Banach space. By the James theorem, there exists a linear functional $x_0^*\in X^*$, which fails to attain the norm. Fix a non-zero function $h$ in $A(\mathbb{D})$. Define $T\in\mathscr{B}(X, A(\mathbb{D}))$ by
\[
Tx=x_0^*(x)h \qquad (x \in X).
\]
By similar computation, as in Example \ref{example of disc}, we conclude that $T$ is not norm attaining and the family
\[
\mathcal{F}=\{x_0^*(x)h: x\in S_X\},
\]
is equicontinuous at every point of $\mathbb{T}$.
\end{example}

Finally, an example of norm attaining operator on $A(\bd)$:

\begin{example}
Let $x^* \in A(\mathbb{D})^*$, $\|x^*\|=1$, be a norm attaining functional. There exists $f_0 \in S_{A(\bd)}$ such that $|x^*(f_0)| = 1$. Fix a nonzero $h \in A(\mathbb{D})$. Define $T \in \mathscr{B}(A(\mathbb{D}))$ by
\begin{align*}
Tf=x^*(f)h,
\end{align*}
for all $f \in A(\mathbb{D})$. A routine computation ensures that
\[
\|T\|=\|x^*\|\|h\|.
\]
Then
\[
\begin{split}
\|Tf_0\|& =\|x^*(f_0)h\|
\\
& =|x^*(f_0)|\|h\|
\\
& =\|x^*\|\|h\|
\\
&=\|T\|,
\end{split}
\]
implies that $T$ is norm attaining.
\end{example}

The above example, in particular, applies to evaluation functionals. For a fixed $z_0 \in \bd$, define the evaluation functional $ev_{z_0} \in A(\bd)^*$ by
\[
ev_{z_0} f = f(z_0) \qquad (f \in A(\bd)).
\]
Clearly
\[
\|ev_{z_0}\|=1=|ev_{z_0}(1)|,
\]
that is, $ev_{z_0}$ is norm attaining. Moreover, $\|T1\|=\|T\|$, where $T$ is defined as in the above example.

Now we turn to the operator ideals of $\clb(A(\mathbb{D}))$. Given a Banach space $X$, a subset $\mathscr{I} \subseteq \mathscr{B}(X)$ is an \textit{operator ideal} if $\mathscr{I}$ contains all operators of finite rank and
\[
T_1\circ T \circ T_2 \in \mathscr{I},
\]
for all $T \in \mathscr{I}$ and $T_1,T_2 \in \mathscr{B}(X)$. The proof of the subsequent corollary is, in essence, similar to that of \cite[Corollary 4.1]{Bala}.

\begin{corollary}
Let $\mathscr{I} \subseteq \clb(A(\mathbb{D}))$ be an operator ideal, $T \in \mathscr{I}$, $\|T\|=1$, $f_0\in S_{A(\mathbb{D})}$, $\epsilon \in (0,1)$, and let $\theta_0$ be a real number. Suppose $\mathcal{F}=\{\widetilde{Tg}: g \in S_{A(\mathbb{D})}\}$ is equicontinuous at $e^{i\theta_0}$, and
$$
\left|\widetilde{Tf_0}(e^{i\theta_0})\right|>1-\dfrac{\epsilon}{3}.
$$
Then there exist $N \in \mathscr{I}$ and $g_0\in S_{A(\mathbb{D})}$ such that
\begin{enumerate}
\item $\|N g_0\|=\|N\|=1$,
\item $\|f_0 - g_0\|< \sqrt{2\epsilon}$, and
\item $\|T - N\|<8\epsilon$.
\end{enumerate}
\end{corollary}
\begin{proof}
By applying Corollary \ref{cor} to $T$, it is enough to show that $N \in \mathscr{I}$, where $N$ is defined as in \eqref{eqn: N}:
\[
(Ng)(z):= \eta(z)\Psi_2(g) + (1-\epsilon)(1-\eta(z))Tg(z),
\]
for all $g \in A(\bd)$ and $z \in \bd$. We write
\[
N = N_1 + N_2,
\]
where
\[
(N_1g)(z)=\eta(z)\Psi_2(g),
\]
and
\[
(N_2g)(z)=(1-\epsilon)(1-\eta(z))Tg(z),
\]
for all $g \in A(\bd)$ and $z \in \bd$. Since $\mathscr{I}$ is an ideal, it is enough to prove that $N_1, N_2 \in \mathscr{I}$. We see that $N_1 \in \mathscr{I}$ as $N_1$ is a rank one operator. Consequently, $ T \in \mathscr{I}$ gives us $N_2 \in \mathscr{I}$, and completes the proof.
\end{proof}
	
The same result holds for weakly compact operators as in \cite{Bala}.

\vspace{0.1in}
	
\noindent\textbf{Acknowledgement:}
The research of first author is supported by the Theoretical Statistics and Mathematics Unit, Indian Statistical Institute, Bangalore, India.
The research of the second named author is supported in part by TARE (TAR/2022/000063), SERB, Department of Science \& Technology (DST), Government of India. The research of the third named author is supported by the NBHM postdoctoral fellowship, Department of Atomic Energy (DAE), Government of India (File No: 0204/3/2020/R$\&$D-II/2445).

\end{document}